\documentclass[11pt,letterpaper]{article}
\usepackage{amsmath,amssymb,amsthm,eucal}
\usepackage{hyperref}
\usepackage{enumerate}
\DeclareMathOperator{\Tr}{Tr}
\DeclareMathOperator{\Gal}{Gal}
\newcommand{\C}{{\mathbb C}}
\newcommand{\F}{{\mathbb F}}
\newcommand{\N}{{\mathbb N}}
\newcommand{\Z}{{\mathbb Z}}
\newcommand{\card}[1]{\left|{#1}\right|}
\newcommand{\sums}[1]{\sum_{\substack{#1}}}
\newcommand{\Fu}{F^*}
\newcommand{\Fp}{\F_p}
\newcommand{\Ftwo}{\F_2}
\newcommand{\Fptothem}{\F_{p^m}}
\newcommand{\Fptothetwom}{\F_{p^{2 m}}}
\newcommand{\Fptothen}{\F_{p^n}}
\newcommand{\Ftwotothen}{\F_{2^n}}
\newcommand{\Ffour}{\F_4}
\newcommand{\ac}[1]{\overline{#1}}
\newcommand{\acF}{\ac{F}}
\newcommand{\acFu}{\ac{F}^*}
\newcommand{\acFp}{\ac{\F}_p}
\newcommand{\acFpu}{\ac{\F}_p^*}
\newcommand{\acFtwou}{\ac{\F}_2^*}
\newcommand{\cP}{{\mathcal P}}
\newcommand{\cW}{{\mathcal W}}
\newcommand{\cZ}{{\mathcal Z}}

\newtheorem{theorem}{Theorem}[section]
\newtheorem{proposition}[theorem]{Proposition}
\newtheorem{lemma}[theorem]{Lemma}
\newtheorem{conjecture}[theorem]{Conjecture}
\theoremstyle{definition}
\newtheorem{definition}[theorem]{Definition}

\title{The resolution of Niho's last conjecture concerning sequences, codes, and Boolean functions}
\author{Tor Helleseth \thanks{T.~Helleseth and C.~Li are with the Department of Informatics, University of Bergen, Norway. The work of T.~Helleseth and C.~Li was supported by the Research Council of Norway (No.~247742/O70 and No.~311646/O70).
	The work of C.~Li was also supported in part by the National Natural Science Foundation of China under Grant (No.~61771021).} \and Daniel J.~Katz  \thanks{D.~J.~Katz is with the Department of Mathematics, California State University, Northridge, USA. This paper is based upon work of D.~J.~Katz supported in part by the National Science Foundation under Grants DMS-1500856 and CCF-1815487.} \and Chunlei Li$^{*}$ }
\date{08 July 2021}
\begin{document}
\maketitle
\begin{abstract}
A new method is used to resolve a long-standing conjecture of Niho concerning the crosscorrelation spectrum of a pair of maximum length linear recursive sequences of length $2^{2 m}-1$ with relative decimation $d=2^{m+2}-3$, where $m$ is even.  The result indicates that there are at most five distinct crosscorrelation values.  Equivalently, the result indicates that there are at most five distinct values in the Walsh spectrum of the power permutation $f(x)=x^d$ over a finite field of order $2^{2 m}$ and at most five distinct nonzero weights in the cyclic code of length $2^{2 m}-1$ with two primitive nonzeros $\alpha$ and $\alpha^d$.  The method used to obtain this result proves constraints on the number of roots that certain seventh degree polynomials can have on the unit circle of a finite field.  The method also works when $m$ is odd, in which case the associated crosscorrelation and Walsh spectra have at most six distinct values.
\end{abstract}
\section{Introduction}
Binary maximum length linear recursive sequences, or m-sequences for short, are widely employed in navigation, radar, and spread-spectrum communication systems because of their good autocorrelation and crosscorrelation properties.
In this paper $\F_q$ denotes a finite field of order $q$, and if $F$ is a field, then $\Fu$ denotes the group of units of $F$.
Let $n$ and $d$ be two positive integers with $\gcd(2^n-1, d)=1$.
It was already known to Niho \cite[pp.~15--20]{Niho} that the study of the value distribution of the crosscorrelation function between two binary m-sequences of length $2^n-1$ with decimation $d$ is equivalent to the study of the weight distribution of the cyclic codes of length $2^n-1$ with two nonzeros $\alpha, \alpha^d$, where $\alpha$ is a primitive element of $\Ftwotothen$.
Furthermore, although Niho does not explicitly mention Walsh spectra in his thesis, he writes his results on crosscorrelation in terms of the quantity $\Delta_d$ whose formula \cite[p.~2]{Niho} is immediately recognizable as that of the Walsh transform of a Boolean function of the form $x\mapsto \Tr_{\Ftwotothen/\Ftwo}(x^d)$.\footnote{Niho actually uses $r$ where we use $d$.}
The Walsh transform measures the nonlinearity of the component functions of the power permutation $x\mapsto x^d$ over $\Ftwotothen$, so it is of interest in measuring the resistance to linear attacks on cryptographic systems employing this permutation.
For an explicit recognition that all three of these information-theoretic questions constitute the same mathematical problem see \cite[p.~613]{Dobbertin-Felke-Helleseth-Rosendahl}, and for another equivalent problem in finite projective geometry see \cite{Games-Sequences, Games-Quadrics}.
The appendix of \cite{Katz-2012} contains all proofs of the equivalences, with the linking mathematical object being the Weil sum of a binomial, which we describe next.

\begin{definition}[Weil sum $W_{F,d}(a)$]\label{Victor}
Let $F$ be a finite field of characteristic $p$ and order $p^n$, let $\Tr\colon F \to \Fp$ denote the absolute trace $\Tr(x)=x+x^p+\cdots+x^{p^{n-1}}$, let $\psi_F\colon F \to \C$ denote the canonical additive character $\psi_F(x)=\exp(2\pi i \Tr(x)/p)$, and let $d$ be a positive integer with $\gcd(d,p^n-1)=1$.
Then for each $a \in F$, we define the Weil sum
\[
W_{F,d}(a)=\sum_{x \in F} \psi_F(x^d- a x),
\]
with the binomial $x^d- a x$ as its argument.
\end{definition}
From the values of the Weil sum $W_{F,d}(a)$, one can determine the following. 
\begin{itemize}
\item The crosscorrelation spectrum for a pair of m-sequences of length $p^n-1$ with relative decimation $d$ is given by the collection of values $W_{F,d}(a)-1$ as $a$ runs through $\Fu$.
\item The Walsh spectrum of the power permutation $x\mapsto x^d$ over $F$ is given by the collection of values $W_{F,d}(a)$ as $a$ runs through $F$.  One should note that $W_{F,d}(0)=0$ invariably, so one can deduce the Walsh spectrum from the crosscorrelation spectrum, and vice versa.  The Walsh spectrum determines the nonlinearity $NL_{F,d}$ of the Boolean function $x\mapsto \Tr(x^d)$ via the relation $NL_{F,d} = (|F| - \max_{a\in F}|W_{F,d}(a)|)/2$.  The nonlinearity assesses a Boolean function's resistance against linear attacks in cryptographic applications.
\item When $d \equiv 1\pmod{p-1}$ and $d$ is not a power of $p$ modulo $p^n-1$, the cyclic code $C$ of length $p^n-1$ with nonzeros at $\alpha$ (a primitive element of $F$) and $\alpha^d$ has the zero word and $2(p^n-1)$ words of weight $(p-1) p^{n-1}$ from the two simplex codes (one with nonzero $\alpha$ and one with nonzero $\alpha^d$) that lie within $C$, and for each $a \in \Fu$, there are $p^n-1$ additional words of weight $(p-1) p^{n-1}-(p-1) W_{F,d}(a)/p$ in $C$.
\end{itemize}
\noindent One can also determine the autocorrelation and crosscorrelation spectrum for a family of Gold sequences from the crosscorrelation spectrum of the pair of m-sequences we just described; this works both for those particular classes of m-sequences used in Gold's original construction \cite{Gold} and, more generally, for his construction applied to any pair of m-sequences \cite[Section 7.3]{Katz-2019}.
In applications it is of particular interest to find positive integers $d$ that lead to Walsh spectra or crosscorrelation spectra consisting of a few values whose absolute values are small \cite{Niho, Helleseth-1976, Helleseth-1978, Helleseth-Kumar, Canteaut-Charpin-Dobbertin, Hollmann-Xiang}, since Boolean functions with high nonlinearity are resistant to linear attack and sequences pairs with low crosscorrelation are easily distinguishable.
We say that an exponent $d$ is {\it degenerate over $F=\mathbb{F}_{p^n}$} when it is a power of $p$ modulo $p^n-1$; in this case $\Tr(x^d)=\Tr(x)$ and our power permutation is linear and our decimated m-sequence is the same as the original.
In this case the crosscorrelation spectrum degenerates to the autocorrelation spectrum of an m-sequence, which has the value $p^n-1$ at shift $0$ and the value $-1$ at all other shifts; the corresponding Walsh spectrum has a single instance of $p^n$ and all other values equal to $0$.
The first author in \cite[Theorem 4.1]{Helleseth-1976} showed that the crosscorrelation spectrum for two binary m-sequences of length $2^n-1$ with decimation $d$ has at least three values if $d$ is not degenerate.
When $n$ is even, say $n=2 m$, so that $F=\Fptothen$ is the quadratic extension of the field $\Fptothem$, we say that the exponent $d$ is a {\it Niho-type exponent over $F$} if it is degenerate over $\Fptothem$ (i.e., a power of $p$ modulo $p^m-1$) but nondegenerate over $F=\Fptothen$ (i.e., not a power of $p$ modulo $p^n-1=p^{2 m}-1$).
Research has shown that exponents $d$ of Niho-type over $F$ are of great importance in generating few-valued crosscorrelation spectra of m-sequences \cite{Niho} and in constructing other interesting objects, such as (vectorial) bent functions and permutations in cryptography \cite{Li-Zeng}.

If $e \equiv p^k d \pmod{p^n-1}$ for some integer $k$, then the exponent $e$ produces the same crosscorrelation spectrum and the same Walsh spectrum as exponent $d$; see \cite[Theorem 2.4]{Trachtenberg}, \cite[Theorem 3.1(d)]{Helleseth-1976}, \cite[Section 1]{Cusick-Dobbertin}, and \cite[Lemma 3.2]{Aubry-Katz-Langevin}.  Therefore, up to this equivalence one can write a Niho exponent over $F=\Fptothen=\Fptothetwom$ as $s(p^m-1)+1=s(\sqrt{\card{F}}-1)+1$ with $s>1$.\footnote{To make $\gcd(d,p^n-1)=1$, it is necessary and sufficient that $s$ satisfy $\gcd(2 s-1,p^m+1)=1$.  For $p=2$ and $s=2$, this happens if and only if $m$ is even.  For $p=2$ and $s=3$, this happens if and only if $m\not\equiv 2 \pmod{4}$.  For $p=2$ and $s=4$, this happens for all $m$.}
The crosscorrelation spectrum of two binary m-sequences with Niho-type decimations for the case $s=2$ is relatively simple and was settled by Niho in his doctoral thesis \cite[Theorem 3-6]{Niho}.
Niho also showed that the crosscorrelation function for binary m-sequences takes at most six values for $s=3$ \cite[Theorem 3-9]{Niho} when $m$ is odd and takes at most eight values for $s=4$ \cite[Theorem 3-10]{Niho}.
Based on numerical results, he further conjectured that the crosscorrelation can actually take at most five values for $s=4$ when $m$ is even \cite[Conjecture 4-6(5)]{Niho}.
By 2006, Dobbertin et al.~had made significant progress in determining the crosscorrelation spectrum of binary m-sequences for Niho-type decimations with $s=3$ via Dickson polynomials and Kloosterman sums \cite{Dobbertin-Felke-Helleseth-Rosendahl}; and recently Xia et al.~completely determined the value distribution for $s=3$ with arbitrary $m$ (even or odd) via a connection with the binary Zetterberg codes \cite[Theorem 2]{Xia-Li-Zeng-Helleseth}.
For Niho-type decimations with $s>4$, the crosscorrelation spectrum contains at most $2 s$ distinct values (cf.~Lemma \ref{William}).
Numerical results for small integers $s>4$ and small values of $m$ (i.e., in small fields) show spectra with significantly fewer than $2 s$ distinct values, but with a tendency to include more values as $m$ and $s$ increase.

In this paper, we shall look into the conjecture on crosscorrelation for the Niho-type decimation with $s=4$ and even $m$ \cite[Conjecture 4-6(5)]{Niho}, which is the final conjecture of Niho's thesis and has remained an open question for almost half a century.  We state Niho's conjecture in the notation of this paper.
\begin{conjecture}[Niho, 1972]
Let $F$ be a finite field of order $4^m$ where $m$ is even, let $d=4(\sqrt{\card{F}}-1)+1$, and let $W_{F,d}(a)$ be the Weil sum from Definition \ref{Victor}.  Then $\{W_{F,d}(a): a \in \Fu\}$ contains at most five distinct values.
\end{conjecture}
In this paper we verify this conjecture.  In fact, we prove the following.
\begin{theorem}\label{George}
Let $F$ be a finite field of order $4^m$, let $d=4(\sqrt{\card{F}}-1)+1$, let $W_{F,d}(a)$ be the Weil sum from Definition \ref{Victor}, and let $\cW=\{W_{F,d}(a)/\sqrt{\card{F}}: a \in \Fu\}$.
\begin{enumerate}[(i).]
\item If $m$ is even, then $\cW \subseteq \{-1,0,1,2,4\}$.
\item If $m$ is odd and greater than $1$, then $\cW \subseteq \{-1,0,1,2,3,4\}$.
\item If $F=\Ffour$, then $d$ is degenerate over $F$ and $\cW=\{0,2\}$.
\end{enumerate}
\end{theorem}
From this one can determine the possible values in the associated crosscorrelation and Walsh spectra, as well as the weights in the associated cyclic code as described in the list after Definition \ref{Victor}.
Notice that our theorem works both when $m$ is even and when $m$ is odd, and shows that the crosscorrelation actually takes at most five values when $m$ is even and at most six values when $m$ is odd.

This paper is organized as follows.
Section \ref{Raphael} shows that Theorem \ref{George} is equivalent to a problem of counting how many zeros of certain polynomials (called {\it key polynomials}) lie in particular subsets of finite fields (called {\it unit circles}).
Then Section \ref{Herbert} investigates a general symmetry property shared by the key polynomials and shows that a certain transformation (called the {\it conjugate-reciprocal map}) permutes the roots of polynomials with this symmetry.
Section \ref{Robert} describes the action of the group generated by the conjugate-reciprocal map, examines the orbits under this action, and calculates certain sums of symmetric rational functions associated with these orbits.
Section \ref{Katherine} then focuses specifically on the key polynomials to give constraints on how many zeros they may have on unit circles, and uses the sums from Section \ref{Robert} to complete the proof of Theorem \ref{George}.
In the development of our proof we shall formulate the problem in arbitrary characteristic and add the restriction to characteristic $2$ when needed.

\section{An Equivalent Zero-Counting Problem}\label{Raphael}

In this section, we show that proving Theorem \ref{George} is equivalent to proving a result (Theorem \ref{Mary} below) about the number of roots of a family of polynomials on the so-called unit circle of a finite field.
To state the result, we first need a few notational conventions and definitions.
If $F$ is a subfield of $E$, we write $[E:F]$ to denote the degree of $E$ over $F$, so that $\card{E}=\card{F}^{[E:F]}$.
We consider all finite fields of a given characteristic $p$ to lie in a unique algebraic closure of $\Fp$.
For any finite field $F$, we write $\acF$ for this algebraic closure, so if $F$ is of characteristic $p$, then $\acF=\acFp$.
\begin{definition}[Half field]\label{Henry}
Let $F$ be a finite field that is an even degree extension of its prime subfield $\Fp$.  Then the {\it half field of $F$}, denoted $H_F$ is unique subfield of $F$ with $[F:H_F]=2$.
\end{definition}
That is, the half field $H_F$ is the unique subfield of $F$ with cardinality $\sqrt{\card{F}}$.
\begin{definition}[Conjugation map]
Let $F$ be a finite field that is an even degree extension of its prime subfield $\Fp$. Then the {\it conjugation map for $F$}, denoted $\tau_F\colon \acF \to \acF$, is the map with $\tau_F(x)=x^{\card{H_F}}$ for every $x \in \acF$.
\end{definition}
If $\tau_F$ is restricted to $F$, one obtains the unique generator of the Galois group $\Gal(F/H_F)$, which is a cyclic group of order $2$.
Note that if $E$ is an extension of $F$, then $\tau_E=\tau_F^{[E:F]}$.
\begin{definition}[Conjugate-reciprocal map]\label{Clarence}
Let $F$ be a finite field that is an even degree extension of its prime subfield $\Fp$. Then the {\it conjugate-reciprocal map for $F$}, denoted $\pi_F\colon \acFu \to \acFu$, is the map with $\pi_F(x)=x^{-\card{H_F}}$ for every $x \in \acFu$.
\end{definition}
The name of $\pi_F$ comes from the fact that $\pi_F(x)=\tau_F(1/x)=1/\tau_F(x)$ for every $x \in \acFu$.
We note that $\pi_F^2=\tau_F^2$.
Therefore, if $E$ is an odd degree extension of $F$, then $\pi_E=\pi_F^{[E:F]}$, but if $E$ is an even degree extension of $F$, then $\pi_E(x)=1/\pi_F^{[E:F]}(x)$ for every $x \in \acFu$.
In particular, if $r \in \acFu$ with $e=[F(r):F]$, then $\pi_F^{2 e}(r)=\tau_F^{2 e}(r)=\tau_{F(r)}^2(r)=r$.  This shows that $\pi_F$ is a permutation of $\acFu$, and so we can write $\pi_F^k$ for both positive and negative integers $k$.
\begin{definition}[Unit circle]\label{Ursula}
Let $F$ be a finite field that is an even degree extension of its prime subfield $\Fp$.  Then the {\it unit circle of $F$}, denoted $U_F$, is the set $\{u \in \acF: u\tau_F(u)=1\}$.
\end{definition}
Note that $U_F$ is the cyclic group of order $\card{H_F}+1$ in $\acF$, and $U_F$ is in fact a subgroup of the cyclic group $\Fu$ since $\card{H_F}+1=\sqrt{\card{F}}+1$ is a divisor of $\card{F}-1$.
Equivalent definitions of $U_F$ include $\{u \in \acFu: \pi_F(u)=u\}$, $\{u \in F: u\tau_F(u)=1\}$, and $\{u \in \Fu: \pi_F(u)=u\}$.
Note that if $E$ is an odd degree extension of $F$, then $U_F \subseteq U_E$.
If $E$ is an even degree extension of $F$, and these fields are of characteristic $2$, then one can show that $U_F\cap U_E=\{1\}$ because $\gcd(\card{H_E}+1,\card{H_F}+1)=1$.

The proof of the equivalence between Theorem \ref{George} and a zero-counting problem goes back to Niho's thesis: see Theorem 3-5 and its proof in \cite{Niho}.
A generalization of Niho's result was stated in \cite[Theorem 2]{Rosendahl}; we now state and prove a corrected\footnote{When $p=3$, $n=2$, and $d=5$, \cite[Theorem 2]{Rosendahl} asserts that $C_5(1)=5$ ($W_{\F_9,5}(1)=6$ in our notation) but a direct calculation shows that $C_5(1)=2$ ($W_{\F_9,5}(1)=3$ in our notation).} version.
\begin{lemma}\label{William}
Let $F$ be a finite field that is an even degree extension of its prime subfield, let $s$ be a nonnegative integer, let $d=s(\sqrt{\card{F}}-1)+1$, let $W_{F,d}(a)$ be the Weil sum from Definition \ref{Victor}, and for each $a \in F$ let $Z(a)$ be the number of distinct zeros of $x^{2s-1}-a x^s-\tau_F(a) x^{s-1} + 1$ that lie on $U_F$.
Then $W_{F,d}(a)=(Z(a)-1)\sqrt{\card{F}}$ for each $a \in F$.
\end{lemma}
\begin{proof}
Let $\alpha$ be a primitive element of $F$, let $q=\card{H_F}=\sqrt{\card{F}}$, and let $Y=\{\alpha^0,\alpha^1,\ldots,\alpha^q\}$.
Then each element of $\Fu$ is uniquely represented as $h y$ for some $h \in H_F^*$ and $y \in Y$, and $h^d=h$ for every $h \in H_F$, so we can write our Weil sum as $W_{F,d}(a) = 1 + \sum_{y \in Y} \sum_{h \in H_F^*} \psi_F((h y)^d - a h y) = 1-\card{Y} + \sum_{y \in Y} \sum_{h \in H_F} \psi_{H_F} (h \Tr_{F/H_F}(y^d-a y)) = q (N(a)-1)$, where $N(a)$ is the number of $y \in Y$ with $\Tr_{F/H_F}(y^d-a y)=0$.
Now note that $\Tr_{F/H_F}(y^d-a y)=y^d-a y+\tau_F(y^d-a y)=y^d-a y+y^{q d}-\tau_F(a) y^q$.  Since $Y \subseteq \Fu$ and $y^{q d}=y^{s(q^2-q)+q}=y^{-(s-1)(q-1)+1}$ for any $y \in \Fu$, our $N(a)$ counts the number of $y \in Y$ with $y^{-(s-1)(q-1)+1} -a y - \tau_F(a) y^{(q-1)+1} + y^{s(q-1)+1}=0$, which (by dividing by $y^{s(q-1)+1}$) is the same as the number of $y \in Y$ with $y^{-(2 s-1)(q-1)}-a y^{-s(q-1)} - \tau_F(a) y^{-(s-1)(q-1)} + 1 =0$.  The power function $y \mapsto y^{-(q-1)}$ maps $Y$ bijectively to $U_F$, so $N(a)$ counts the number of $x \in U_F$ such that $x^{2 s-1} -a x^s -\tau_F(a) x^{s-1} +1=0$, i.e., $N(a)=Z(a)$.
\end{proof}
This lemma also shows that as $a$ runs through $\Fu$, the number of polynomials $x^{2 s-1}-a x^s -\tau_F(a) x+1$ that have $r$ distinct roots on $U_F$ is the same as the number of times $W_{F,d}(a)$ assumes the value $(r-1)\sqrt{\card{F}}$.

Now Lemma \ref{William} shows that proving Theorem \ref{George} is equivalent to proving the following result.
\begin{theorem}\label{Mary}
Let $F$ be a finite field that is an extension of $\Ffour$, let $d=4(\sqrt{\card{F}}-1)+1$, and for each $a \in \Fu$ let $g_{F,a}(x)=x^7-a x^4-\tau_F(a) x^3 + 1$ and let $Z(a)$ be the number of distinct roots of $g_{F,a}(x)$ that lie in $U_F$.  Let $\cZ=\{Z(a): a \in \Fu\}$.
\begin{enumerate}[(i).]
\item If $[F:\Ffour]$ is even, then $\cZ \subseteq \{0,1,2,3,5\}$.
\item If $[F:\Ffour]$ is odd and greater than $1$, then $\cZ \subseteq \{0,1,2,3,4,5\}$.
\item If $F=\Ffour$, then $d$ is degenerate over $F$ and $\cZ=\{1,3\}$.
\end{enumerate}
\end{theorem}
We now see that our problem is tantamount to counting the zeros of certain polynomials on unit circles, so we give a special name to these polynomials.
\begin{definition}[Key polynomial]\label{Alan}
If $F$ is a finite field of even degree over its prime subfield $\Fp$ and $a \in F$, then the {\it key polynomial for $a$ over $F$}, written $g_{F,a}(x)$, is the polynomial
\[
g_{F,a}(x)=x^7- a x^4 - \tau_F(a) x^3 + 1
\]
in $F[x]$.
\end{definition}

\section{Self-conjugate-reciprocal polynomials and their roots}\label{Herbert}

This section explores the properties of a class of polynomials that includes the key polynomials $g_{F,a}(x)$ whose roots on the unit circle we must count to prove Theorem \ref{Mary}.
\begin{definition}[Reciprocal of a polynomial]
If $F$ is a field, and $f(x)=f_0 + f_1 x + \cdots + f_d x^d \in F[x]$ with $f_d\not=0$, then the {\it reciprocal of $f(x)$}, written $f^*(x)$, is the polynomial $x^d f(1/x)=f_d + f_{d-1} x + \cdots +f_0 x^d$.  We decree that the reciprocal of the zero polynomial is the zero polynomial.
\end{definition}
We note that if either $f(x)=0$ or $f(x)$ has a nonzero constant coefficient, then $f^{**}(x)=f(x)$, but this is not true if both $f(x)\not=0$ and $f(0)=0$.  Also note that if $h(x)=f(x) g(x)$, then $h^*(x)=f^*(x) g^*(x)$.
\begin{definition}[Conjugate of a polynomial]
If $F$ is a finite field that is an even degree extension of its prime subfield and $f(x) \in \acF[x]$, then the {\it conjugate of $f(x)$ over $F$}, written $f^{\tau_F}(x)$, is the polynomial $\tau_F(f_0) + \tau_F(f_1) x + \cdots + \tau_F(f_d) x^d$.
\end{definition}
We note that $f^{\tau_F \tau_F}(x)=f(x)$ for every $f(x) \in F[x]$.  Also note that if $h(x)=f(x) g(x)$, then $h^{\tau_F}(x)=f^{\tau_F}(x) g^{\tau_F}(x)$.  If $E$ is an extension of $F$ and $f(x) \in F[x]$, then $f^{\tau_E}(x)=f^{\tau_F}(x)$ if $[E:F]$ is odd, but $f^{\tau_E}(x)=f(x)$ if $[E:F]$ is even.
\begin{definition}[Conjugate-reciprocal of a polynomial]\label{Esther}
If $F$ is a finite field that is an even degree extension of its prime subfield and $f(x) \in \acF[x]$, then the {\it conjugate-reciprocal of $f(x)$ over $F$}, written $f^{*\tau_F}(x)$, is the conjugate over $F$ of the reciprocal of $f(x)$.
\end{definition}
We note that the reciprocal and conjugate operations commute, i.e., $f^{*\tau_F}(x)=f^{\tau_F *}(x)$ for every $f(x) \in F[x]$.  If either $f(x)=0$ or $f(x) \in F[x]$ has a nonzero constant coefficient, then $f^{*\tau_F*\tau_F}(x)=f(x)$, but this is not true if both $f(x)\not=0$ and $f(0)=0$.  Also note that if $h(x)=f(x) g(x)$, then $h^{*\tau_F}(x)=f^{*\tau _F}(x) g^{*\tau_F}(x)$.  If $E$ is an extension of $F$ and $f(x) \in F[x]$, then $f^{*\tau_E}(x)=f^{*\tau_F}(x)$ if $[E:F]$ is odd, but $f^{*\tau_E}(x)=f^*(x)$ if $[E:F]$ is even.
\begin{definition}
If $F$ is a finite field that is an even degree extension of its prime subfield, a {\it self-conjugate-reciprocal polynomial over $F$} is a polynomial $f(x) \in \acF[x]$ that is its own conjugate-reciprocal over $F$, i.e., $f(x)=f^{*\tau_F}(x)$.
\end{definition}
Note that $0$ is self-conjugate-reciprocal, but any nonzero self-conjugate reciprocal polynomial must have a nonzero constant coefficient. If $E$ is an odd degree extension of $F$, then any self-conjugate-reciprocal polynomial over $F$ is also a self-conjugate-reciprocal polynomial over $E$.

The key polynomials $g_{F,a}(x)$ of Definition \ref{Alan}, whose roots on $U_F$ we must count to prove Theorem \ref{Mary}, are self-conjugate-reciprocal over $F$.
The rest of this section is dedicated to understanding the relation between the conjugate-reciprocal operation on polynomials from Definition \ref{Esther} and the conjugate-reciprocal map $\pi_F$ from Definition \ref{Clarence}.
\begin{lemma}\label{Albert}
Let $F$ be a finite field that is an even degree extension of its prime subfield $\Fp$, let $f(x)$ be a nonzero polynomial in $F[x]$, and let $r \in \acFpu$.
Then $r$ is a root of $f(x)$ if and only if $\pi_F(r)$ is a root of $f^{*\tau_F}(x)$.
\end{lemma}
\begin{proof}
Write $f(x)=\sum_{j=0}^d f_j x^j$ with $f_d\not=0$.
Note that  $\pi_F(r)=1/\tau_F(r)$ exists and is nonzero because $r\not=0$ and $\tau_F$ is an automorphism of $\acFp$.
Then we have
\begin{align*}
f^{*\tau_F}(\pi_F(r))
& = \sum_{k=0}^d \tau_F(f_{d-k}) \pi_F(r)^k \\
& = \sum_{j=0}^d \tau_F(f_j) \pi_F(r)^{d-j} \\
& = \sum_{j=0}^d \tau_F(f_j) \tau_F(r)^{j-d} \\
& = \tau_F(r)^{-d} \tau_F(f(r)),
\end{align*}
from which we can see that $\pi_F(r)$ is a root of $f^{*\tau_F}(x)$ if and only if $r$ is a root of $f(r)$.
\end{proof}
We now sharpen the correspondence in Lemma \ref{Albert} to show that multiplicities of roots are respected.
\begin{lemma}\label{Eric}
Let $F$ be a finite field that is an even degree extension of its prime field $\Fp$.  Let $f(x)$ be a nonzero polynomial in $F[x]$.  If $r$ is a root of $f(x)$ in $\acFpu$ with multiplicity $m$, then $\pi_F(r)$ is a root of $f^{*\tau_F}(x)$ in $\acFp$ with multiplicity $m$.
\end{lemma}
\begin{proof}
Since $r$ is a root with multiplicity $m$, we can write
\[
f(x)=(x-r)^m g(x)
\]
for some $g(x) \in \acFp[x]$ with $r$ not a root of $g(x)$.  Now take the conjugate-reciprocal of both sides over $F$ to obtain
\begin{align*}
f^{*\tau_F}(x)
& =(1-\tau_F(r) x)^m g^{*\tau_F}(x) \\
& =(x-\pi_F(r))^m (-\tau_F(r))^m g^{*\tau_F}(x),
\end{align*}
and since Lemma \ref{Albert} shows that $\pi_F(r)$ is not a root of $g^{*\tau_F}(x)$, we see that $\pi_F(r)$ is a root of $f^{*\tau_F}(x)$ of multiplicity $m$.
\end{proof}
Now we apply our results to self-conjugate-reciprocal polynomials.
\begin{proposition}\label{Helen}
Let $F$ be a finite field that is an even degree extension of its prime field $\Fp$, and let $f(x)$ be a nonzero self-conjugate-reciprocal polynomial over $F$.
If $r \in \acFp$ is a root of $f(x)$ of multiplicity $m$, then $r\not=0$ and $\pi_F(r)$ is also a root of $f(x)$ of multiplicity $m$.
\end{proposition}
\begin{proof}
A nonzero self-conjugate-reciprocal polynomial must have a nonzero constant coefficient, so $r\not=0$, and then we may apply Lemma \ref{Eric}.
\end{proof}
Proposition \ref{Helen} shows that the roots of a self-conjugate-reciprocal polynomial over $F$ can be organized into orbits under the action of the group of transformations generated by $\pi_F$, with each element in an orbit having the same multiplicity.  We study this group of transformations in the next section.

\section{Action of the Conjugate-Reciprocal Group}\label{Robert}

Throughout this section, we shall use Definitions \ref{Henry}--\ref{Ursula} (for the half field, conjugation map, conjugate-reciprocal map, and unit circle) from Section \ref{Raphael} along with their associated notations.
\begin{definition}[Conjugate-reciprocal group]
Let $F$ be a finite field that is an even degree extension of its prime subfield $\Fp$.  Then the {\it conjugate-reciprocal group for $F$}, denoted $\Pi_F$, is the cyclic group of permutations of $\acFpu$ generated by the conjugate-reciprocal map $\pi_F$ from Definition \ref{Clarence}.
\end{definition}
We are interested in the orbits under $\Pi_F$ of elements of $\acFu$.
\begin{definition}[Orbit of $\Pi_F$]
Let $F$ be a finite field that is an even degree extension of its prime subfield $\Fp$ and $r \in \acFpu$.  Then we denote the orbit of $r$ under the action of the group $\Pi_F$ on $\acFpu$ by $\Pi_F\cdot r=\{\pi_F^k(r): k \in \Z\}$.
\end{definition}
\begin{definition}[$\Pi_F$-Closed]
If $F$ is a finite field that is an even degree extension of its prime subfield $\Fp$ and $R \subseteq \acFpu$, then we say that $R$ is {\it closed under the action of $\Pi_F$} (or simply {\it $\Pi_F$-closed}) to mean that $\Pi_F\cdot r \subseteq R$ for every $r \in R$.
\end{definition}
Equivalently, a $\Pi_F$-closed subset of $\acFpu$ is a union of $\Pi_F$-orbits.
In Subsection \ref{Samuel} we study the size of these orbits, then in Subsection \ref{Jesse} we study certain symmetric rational functions evaluated on $\Pi_F$-closed sets, and in Subsection \ref{Theodore} we compute the traces of those sums when our fields are of characteristic $2$.

\subsection{Sizes of Orbits}\label{Samuel}

Our first task is to try to understand the size of a $\Pi_F$-orbit.
\begin{lemma}\label{David}
Let $F$ be a finite field that is an even degree extension of its prime subfield $\Fp$, let $r \in \acFpu$, and let $e=[F(r):F]$.  Then any $s \in \Pi_F\cdot r$ has the property that $F(s)=F(r)$.  Furthermore,
\begin{enumerate}[(i).]
\item $\card{\Pi_F\cdot r} =e$ if and only if $e$ is odd and $r \in U_{F(r)}$; and
\item $\card{\Pi_F\cdot r} =2 e$ otherwise.
\end{enumerate}
\end{lemma}
\begin{proof}
First of all, note that if $a \in \acFpu$, then $\pi_F(a)=a^{-\card{H_F}}$ and $\pi_F^{-1}(a)=a^{-\card{F(a)}/\card{H_F}}$ are also in $F(a)$, since they are powers of $a$.
This shows that if $s \in \Pi_F\cdot r$, say $s=\pi_F^j(r)$, then $r=\pi_F^{-j}(s)$, so that $s \in F(r)$ and $r \in F(s)$, and so $F(r)=F(s)$.

Since $\pi_F^{2 e}(r)=\tau_F^{2 e}(r)=\tau_{F(r)}^2(r)=r$, we see that $\card{\Pi_F\cdot r}$ is a divisor of $2 e$.  
Furthermore, we cannot have $\pi_F^{2 k}(r)=r$ when $0 < k < e$, because that would mean that $\tau_F^{2 k}(r)=r$, which would mean that $\tau_E^2(r)=r$ for the $k$th degree extension $E$ of $F$, which would place $r$ in $E$ so that $e=[F(r):F] \leq [E:F] = k < e$.  Also we cannot have $\pi_F^k(r)=r$ when $0 < k < e$, for then $\tau_F^{2 k}(r)=\pi_F^{2 k}(r)=r$, contradicting what we just said.  Thus $\card{\Pi_F\cdot r}$ is a divisor of $2 e$ and is greater than or equal to $e$, so it is either $e$ or $2 e$.  Furthermore, if $e$ is even, then $\pi_F^e(r)=r$ would violate the principle that no positive even power of $\pi_F$ less than $2 e$ fixes $r$.  Thus we conclude that $\card{\Pi_F\cdot r}=2 e$ when $e$ is even.  When $e$ is odd, we note that the condition $\pi_F^e(r)=r$ is equivalent to $\pi_{F(r)}(r)=r$, which is equivalent to $r \in U_{F(r)}$, so $\card{\Pi_F\cdot r} = e$ if and only if $r \in U_{F(r)}$.
\end{proof}
Note that Lemma \ref{David} shows that you can determine where $r$ lies by looking at $n=\card{\Pi_F\cdot r}$: if $n$ is odd, then $[F(r):F]=n$ and $r \in U_{F(r)}$; but if $n$ is even, then $[F(r):F]=n/2$, and if we also know that $n\equiv 2 \pmod{4}$ then we can conclude that $r \not\in U_{F(r)}$.

\subsection{Sums on $\Pi_F$-orbits over $\acFpu$}\label{Jesse}

Now we prove some technical results that will be used in the next subsection.
\begin{lemma}\label{Matilda}
Let $F$ be a finite field that is an even degree extension of its prime subfield $\Fp$. Let $R$ be a finite $\Pi_F$-closed subset of $\acFpu$, and let
\[
S=\sums{\{u,v\} \subseteq R \\ u\not=v} \frac{u v}{(u-v)^2}.
\]
Then $S \in H_F$.
\end{lemma}
\begin{proof}
We have
\begin{align*}
\tau_F(S)
& = \sums{\{u,v\} \subseteq R \\ u\not=v} \frac{\tau_F(u) \tau_F(v)}{(\tau_F(u)-\tau_F(v))^2} \\
& = \sums{\{u,v\} \subseteq R \\ u\not=v} \frac{\pi_F(u)^{-1} \pi_F(v)^{-1}}{(\pi_F(u)^{-1}-\pi_F(v)^{-1})^2} \\
& = \sums{\{u,v\} \subseteq R \\ u\not=v} \frac{\pi_F(u) \pi_F(v)}{(\pi_F(u)-\pi_F(v))^2},
\end{align*}
and since $R$ is closed under the action of $\Pi_F$, the map $u \mapsto \pi_F(u)$ is a permutation of $R$, and thus $\{u,v\} \mapsto \{\pi_F(u),\pi_F(v)\}$ is a permutation of the unordered pairs in $R$, and so we may reparameterize the last sum by dropping the maps $\pi_F$ to see that $\tau_F(S) = S$, and hence $S \in H_F$.
\end{proof}

\begin{lemma}\label{James}
Let $F$ be a finite field that is an even degree extension of its prime subfield $\Fp$.
Let $Q,R$ be disjoint finite $\Pi_F$-closed subsets of $\acFpu$, and let
\[
S=\sum_{(u,v) \in Q \times R} \frac{u v}{(u-v)^2}.
\]
Then $S \in H_F$.
\end{lemma}
\begin{proof}
\begin{align*}
\tau_F(S)
& = \sum_{(u,v) \in Q \times R} \frac{\tau_F(u) \tau_F(v)}{(\tau_F(u)-\tau_F(v))^2} \\
& = \sum_{(u,v) \in Q \times R} \frac{\pi_F(u)^{-1} \pi_F(v)^{-1}}{(\pi_F(u)^{-1}-\pi_F(v)^{-1})^2} \\
& = \sum_{(u,v) \in Q \times R} \frac{\pi_F(u) \pi_F(v)}{(\pi_F(u)-\pi_F(v))^2},
\end{align*}
and note that $(u,v)\mapsto (\pi_F(u),\pi_F(v))$ is a permutation of $Q \times R$ since $Q$ and $R$ are closed under the action of $\Pi_F$.
So we may reparameterize the last sum by dropping the maps $\pi_F$ to see that $\tau_F(S) = S$, and hence $S \in H_F$.
\end{proof}
With the sums in Lemmata \ref{Matilda} and \ref{James} known to be in the half field $H_F$, the following subsection further examines their absolute traces in the case of $p=2$.

\subsection{Sums on $\Pi_F$-orbits over $\acFtwou$}\label{Theodore}
We continue with a few more technical results.  Lemmata \ref{Theresa} and \ref{Ellen} are used to prove Proposition \ref{Timothy}, which is the key to the proof of Theorem \ref{Mary}.
\begin{lemma}\label{Theresa}
If $F$ is a finite field that is an even degree extension of $\Ftwo$ and $r \in \acFtwou$, and
\[
S=\sums{\{u,v\} \subseteq \Pi\cdot r \\ u\not=v} \frac{u v}{(u-v)^2},
\]
then $S$ belongs to $H_F$ and
\[
\Tr_{H_F/\Ftwo}(S)=\binom{\card{\Pi_F\cdot r}-1}{2} \pmod{2}.
\]
\end{lemma}
\begin{proof}
Lemma \ref{Matilda} shows that $S \in H_F$.
Let $n=\card{\Pi_F\cdot r}$.
Write $r_j=\pi_F^j(r)$ for every $j \in \Z/n\Z$, so that $\Pi_F\cdot r=\{r_0,r_1,\ldots,r_{n-1}\}$.
For any distinct $j,k \in \Z/n\Z$, define
\[
S_{j,k} = \frac{r_j r_k}{(r_j-r_k)^2}.
\]
Then note that
\[
S_{j,k} = -\frac{r_j}{r_j-r_k} + \left(\frac{r_j}{r_j-r_k}\right)^2 
\]
and since our field is of characteristic $2$, we have
\begin{align*}
\begin{split}
S_{j,k}+S_{j,k}^2+\cdots+S_{j,k}^{\card{H_F}/2}
& = \frac{r_j}{r_j-r_k} + \tau_F\left(\frac{r_j}{r_j-r_k}\right)\\
& = \frac{r_j}{r_j-r_k} + \frac{r_{j+1}^{-1}}{r_{j+1}^{-1}-r_{k+1}^{-1}}\\
& = \frac{r_j}{r_j-r_k} + \frac{r_{k+1}}{r_{k+1}-r_{j+1}}.
\end{split}
\end{align*}
Then
\[
S = \sums{\{j,k\} \subseteq \Z/n\Z \\ j\not=k} S_{j,k},
\]
and so
\[
\Tr_{H_F/\Ftwo}(S) = \sums{\{j,k\} \subseteq \Z/n\Z \\ j\not=k} \left(\frac{r_j}{r_j-r_k} + \frac{r_{k+1}}{r_{k+1}-r_{j+1}}\right).
\]
To help us compute this sum, we put an ordering $0 < 1 < \ldots < n-1$ on $\Z/n\Z$ to obtain
\[
\Tr_{H_F/\Ftwo}(S) = \sum_{0 \leq j < k < n} \frac{r_j}{r_j-r_k} + \sum_{0 \leq j < k < n} \frac{r_{k+1}}{r_{k+1}-r_{j+1}}.
\]
Now the terms with $j=0$ in the first sum are identical to the terms with $k=n-1$ in the second, and since our field is of characteristic $2$, we can drop them to obtain
\[
\Tr_{H_F/\Ftwo}(S) = \sum_{1 \leq j < k < n} \frac{r_j}{r_j-r_k} + \sum_{0 \leq j < k < n-1} \frac{r_{k+1}}{r_{k+1}-r_{j+1}},
\]
and then we note that the pair $(j+1,k+1)$ in the second sum runs through the same set of values as $(j,k)$ in the first, so we can reparameterize the second sum and combine with the first to obtain
\[
\Tr_{H_F/\Ftwo}(S) = \sum_{1 \leq j < k < n} \left(\frac{r_j}{r_j-r_k} + \frac{r_k}{r_k-r_j}\right),
\]
in which every summand equals $1$, and the number of summands is the number of unordered pairs in $\{1,\ldots,n-1\}$.
\end{proof}

\begin{lemma}\label{Ellen}
If $F$ is a finite field that is an even degree extension of $\Ftwo$ and $r,s \in \acFtwou$ belong to different $\Pi_F$-orbits, and
\[
S=\sum_{(u,v) \in \Pi_F\cdot r \times \Pi_F \cdot s} \frac{u v}{(u-v)^2},
\]
then $S$ belongs to $H_F$ and
\[
\Tr_{H_F/\Ftwo}(S)=\card{\Pi_F\cdot r} \card{\Pi_F\cdot s} \pmod{2}.
\]
\end{lemma}
\begin{proof}
Lemma \ref{James} shows that $S \in H_F$.
For any $u \in \Pi_F\cdot r$ and $v \in \Pi_F\cdot s$ we define
\[
S_{u,v} = \frac{u v}{(u-v)^2}
\]
Then note that
\[
S_{u,v} = -\frac{u}{u-v} + \left(\frac{u}{u-v}\right)^2 
\]
and since our field is of characteristic $2$, we have
\begin{align*}
\begin{split}
S_{u,v}+S_{u,v}^2+\cdots+S_{u,v}^{\card{H_F}/2}
& = \frac{u}{u-v} + \tau_F\left(\frac{u}{u-v}\right)\\
& = \frac{u}{u-v} + \frac{\tau_F(u)}{\tau_F(u)-\tau_F(v)}\\
& = \frac{u}{u-v} + \frac{\pi_F(u)^{-1}}{\pi_F(u)^{-1}-\pi_F(v)^{-1}}\\
& = \frac{u}{u-v} + \frac{\pi_F(v)}{\pi_F(v)-\pi_F(u)}.
\end{split}
\end{align*}
Then
\[
S=\sum_{(u,v) \in \Pi_F\cdot r \times \Pi_F \cdot s} S_{u,v},
\]
and so
\[
\Tr_{H_F/\Ftwo}(S) = \sum_{(u,v) \in \Pi_F\cdot r \times \Pi_F \cdot s} \frac{u}{u-v} + \sum_{(u,v) \in \Pi_F\cdot r \times \Pi_F \cdot s} \frac{\pi_F(v)}{\pi_F(v)-\pi_F(u)},
\]
and note that $(u,v)\mapsto (\pi_F(u),\pi_F(v))$ is a permutation of $\Pi_F\cdot r \times \Pi_F\cdot s$ since $\Pi_F\cdot r$ and $\Pi_F\cdot s$ are closed under the action of $\Pi_F$.
So we may reparameterize the second sum by dropping the maps $\pi_F$ and combine with the first sum to obtain
\[
\Tr_{H_F/\Ftwo}(S) = \sum_{(u,v) \in \Pi_F\cdot r \times \Pi_F \cdot s} \left(\frac{u}{u-v}+\frac{v}{v-u}\right),
\]
which is a sum with $\card{\Pi_F\cdot r} \card{\Pi_F\cdot s}$ terms, each equal to $1$.
\end{proof}

\begin{proposition}\label{Timothy}
Let $F$ be a finite field that is an even degree extension of $\Ftwo$.  Let $R$ be the union of $t$ distinct $\Pi_F$-orbits in $\acFtwou$, and let
\[
S=\sums{\{u,v\} \subseteq R \\ u\not=v} \frac{u v}{(u-v)^2}.
\]
Then $S$ belongs to $H_F$ and
\[
\Tr_{H_F/\Ftwo}(S)=\binom{\card{R}+1}{2}+t \pmod{2}.
\]
\end{proposition}
\begin{proof}
Lemma \ref{Matilda} shows that $S \in H_F$.
Let $\cP$ be the partition of $R$ into $\Pi_F$-orbits: so $\cP$ is a set of $t$ distinct $\Pi_F$-orbits, and the union of these orbits is $R$.
Then
\[
S = \sum_{P \in \cP} \sums{\{u,v\} \subseteq P \\ u\not=v} \frac{u v}{(u-v)^2}+\sums{\{P,Q\} \subseteq \cP \\ P\not=Q} \sums{(u,v) \in P \times Q} \frac{u v}{(u-v)^2}.
\]
If we apply $\Tr_{H_F/\Ftwo}$ to $S$, then Lemmata \ref{Theresa} and \ref{Ellen} give the values of traces of the inner sums to yield
\begin{align*}
\Tr_{H_F/\Ftwo}(S)
& = \sum_{P \in \cP} \binom{\card{P}-1}{2} + \sums{\{P,Q\} \subseteq \cP \\ P \not= Q} \card{P}\card{Q} \\
& = \sum_{P \in \cP} \left(1- \card{P} + \binom{\card{P}}{2}\right) + \sums{\{P,Q\} \subseteq \cP \\ P \not= Q} \card{P} \card{Q} \\
& = t - \card{R} + \sum_{P \in \cP} \binom{\card{P}}{2} + \sums{\{P,Q\} \subseteq \cP \\ P \not= Q} \card{P}\card{Q},
\end{align*}
and the last two sums together simply count all pairs of distinct elements in $R$, so we have
\[
\Tr_{H_F/\Ftwo}(S)= t + \card{R} +\binom{\card{R}}{2} \pmod{2}. \qedhere
\]
\end{proof}

\section{The key polynomial}\label{Katherine}

In Section \ref{Raphael} we saw that proving our main result (Theorem \ref{George}) is equivalent to proving Theorem \ref{Mary}, which concerns the numbers of roots on the unit circle of the key polynomials (see Definition \ref{Alan}).
Observe that every key polynomial is a self-conjugate-reciprocal polynomial.  In this section, we prove constraints on the numbers of roots of the key polynomial $g_{F,a}(x)$ on the unit circle $U_F$.  The situation differs somewhat depending on whether or not the key polynomial is separable.  In Subsection \ref{Imogene}, we determine precise conditions on $a$ that make $g_{F,a}(x)$ inseparable, and count the roots of $g_{F,a}(x)$ on $U_F$ in those cases.  In Subsection \ref{Nathan}, we prove constraints on the number of roots of $g_{F,a}(x)$ on $U_F$ in the cases where $g_{F,a}(x)$ is separable.

\subsection{Inseparable key polynomial}\label{Imogene}

First of all, we want to understand when $g_{F,a}(x)$ is separable and when it is not.
\begin{lemma}\label{Cecilia}
Suppose that $F$ is an extension of $\Ffour$ and $a \in F$.
Then the key polynomial $g_{F,a}(x)$ is inseparable if and only if $a \in U_F$.
Furthermore, in the case that $g_{F,a}(x)$ is inseparable, we have the following:
\begin{enumerate}[(1).]
\item\label{Elizabeth} If $a=1$, then $g_{F,a}(x)=(x+1)^5(x^2+x+1)$ has a root of multiplicity $5$ at $1$ and two simple roots at the primitive third roots of unity.
\begin{enumerate}[(a).]
\item If $[F:\Ffour]$ is even, then only the root at $1$ lies on $U_F$, and the other two roots lie in $F\setminus U_F$.
\item If $[F:\Ffour]$ is odd, then all three roots lie on $U_F$.
\end{enumerate}
\item\label{Jane} If $a \in U_F\setminus\{1\}$, then $g_{F,a}(x)=(x^4+1/a)(x^3+a)$ has a root of multiplicity $4$ at $a^{-1/4}$ and three simple roots that are the cube roots of $a$.  The quadruple root always lies on $U_F$.
\begin{enumerate}[(a).]
\item If $[F:\Ffour]$ is even, then precisely one of the three simple roots lies on $U_F$, and the other two simple roots lie in $F\setminus U_F$.  Therefore, a total of two roots of $g_{F,a}(x)$ lie on $U_F$.
\item If $[F:\Ffour]$ is odd, then we have the following:
\begin{enumerate}[(i).]
\item If $a$ is one of the $(\sqrt{\card{F}}-2)/3$ elements that are cubes of elements on $U_F$ (and not equal to $1$), then all four roots of $g_{F,a}(x)$ lie on $U_F$.
\item Otherwise, $a$ is one of the $2(\sqrt{\card{F}}+1)/3$ elements on $U_F$ that are not cubes of elements on $U_F$.  In this case, only the quadruple root lies on $U_F$ and the three simple roots lie in $U_E\setminus U_F$ where $E$ is the extension of $F$ with $[E:F]=3$.
\end{enumerate}
\end{enumerate}
\end{enumerate}
\end{lemma}
\begin{proof}
Let $g(x)=g_{F,a}(x)$, and note that if $a=0$, then $g$ is clearly separable, so we may assume $a\not=0$ henceforth.
It is straightforward to compute that $\gcd(g,g')=\gcd(a x^4+1, x^6+\tau_F(a) x^2)=\gcd(a x^4+1,(\tau_F(a)-1/a) x^2)$, and this is not $1$ if and only if $\tau_F(a)=1/a$, which is equivalent to saying $a \in U_F$.  Furthermore, when this occurs, we see that $\gcd(g,g')=a x^4+1$ and then we note that $g(x)=(x^4+1/a)(x^3+a)$.

Proof of part \eqref{Elizabeth}: First let us examine the case when $a=1$.
Then $g_{F,1}(x)=(x^4+1)(x^3+1)=(x+1)^5(x^2+x+1)$, which has a root of multiplicity $5$ at $1$, which is on $U_F$, and two simple roots at the primitive third roots of unity.
Note that $U_{\Ffour}=\Ffour^*$ is the set of third roots of unity, so all the roots lie in $F$.
If $[F:\Ffour]$ is even, then $U_F\cap U_{\Ffour}=\{1\}$, so $U_F$ does not contain the primitive third roots of unity, and so $g_{F,1}(x)$ has the quintuple root $1$ on $U_F$, but no other roots on $U_F$.
If $[F:\Ffour]$ is odd, then $\card{U_F}=\sqrt{\card{F}}+1$ is divisible by $3$, and so all third roots of unity lie on $U_F$, so all three roots of $g_{F,1}(x)$ lie on $U_F$.

Proof of part \eqref{Jane}: From now on we suppose that $a \in U_F\setminus\{1\}$.
Recall that $g(x)=(x^4+1/a)(x^3+a)$, and one can compute that $\gcd(x^4+1/a,x^3+a)=\gcd(a x+1/a,x^3+a)=\gcd(x+1/a^2,1/a^6+a)$, which is not $1$ if and only if $a^7=1$, which we claim cannot happen.
For $\card{U_F}=\sqrt{\card{F}}+1$, which cannot be divisible by $7$ because $\sqrt{\card{F}}$ is a power of $2$ (hence congruent to $1$, $2$, or $4$ modulo $7$), so $U_F$ cannot have primitive seventh roots of unity, and we have excluded $a=1$ at this point.
So the two factors in our factorization $g(x)=(x^4+1/a)(x^3+a)$ do not share common roots.
The $(x^4+1/a)=(x+a^{-1/4})^4$ factor has a root of multiplicity $4$ at $a^{-1/4}$, which is on $U_F$ since $u \mapsto u^{-1/4}$ is a permutation of $U_F$.
The $(x^3+a)$ factor has three simple roots at the cube roots of $a$.

If $[F:\Ffour]$ is even, then we see that $3 \nmid \card{U_F}$, and therefore $u\mapsto u^3$ is a permutation of $U_F$.  Thus every element of $U_F$ has a unique cube root on $U_F$, and so precisely one of the three simple roots of $g(x)$ lies on $U_F$, which, along with the root of multiplicity $4$, means we have two distinct roots of $g(x)$ on $U_F$.  Since $\Ffour^* \subseteq \Fu$, the other two cube roots of $a$ lie in $F\setminus U_F$.

If $[F:\Ffour]$ is odd, then $U_{\Ffour} \subseteq U_F$, and so the third roots of unity lie on $U_F$.
Thus if $a$ is the cube of an element on $U_F$, then the other two cube roots of $a$ will also lie on $U_F$.  So all or none of the simple roots of $g(x)$ lie on $U_F$.
Since $\card{U_F}$ is divisible by $3$, one-third of the $\sqrt{\card{F}}+1$ elements of $U_F$ are cubes of elements on $U_F$, and therefore $(\sqrt{\card{F}}-2)/3$ of the elements of $U_F\setminus\{1\}$ are cubes of elements on $U_F$.
When $a$ is one of these, all four roots of $g(x)$ lie on $U_F$.
Otherwise $a$ is one of the $2(\sqrt{\card{F}}+1)/3$ elements of $U_F$ that is not a cube of an element on $U_F$, and only the quadruple root at $a^{-1/4}$ lies on $U_F$, and we claim that none of the three simple roots lies in $F$. In fact, all three simple roots must lie in the same $\Pi_F$-orbit of size $3$, since otherwise at least one root would need to be in a singleton orbit, which would place it on $U_F$, contradicting our assumption that $a$ is not the cube of an element of $U_F$.  By Lemma \ref{David}, this means that all three simple roots lie on $U_E$ where $E$ is the extension of $F$ with $[E:F]=3$.
\end{proof}

\subsection{Separable key polynomial}\label{Nathan}

Now we investigate how many roots a separable key polynomial can have on the unit circle.
\begin{lemma}\label{Lawrence}
Let $F$ be a finite field that is an even degree extension of $\Ftwo$, and let $a \in F$.  Suppose that the key polynomial $g_{F,a}(x)$ from Definition \ref{Alan} is separable and $R$ is its set of seven distinct roots in $\acFtwou$.  Let
\[
S=\sums{\{u,v\} \subseteq R \\ u\not=v} \frac{u v}{(u-v)^2}.
\]
Then $S=0$.
\end{lemma}
\begin{proof}
Since $g_{F,a}(x)$ is self-conjugate-reciprocal, the set $R$ of roots of $g_{F,a}(x)$ is closed under the action of $\Pi_F$.
Consider the following polynomials in $\Ftwo[x_1,\ldots,x_7]$:
\[
b(x)=\prod_{1 \leq i < j \leq 7} (x_i-x_j),
\]
and
\[
c(x)=b(x)^2 \sums{1 \leq i < j \leq 7} \frac{x_i x_j}{(x_i-x_j)^2}.
\]
Write $R=\{r_1,\ldots,r_7\}$ so that
\[
S=\frac{c(r_1,\ldots,r_7)}{b(r_1,\ldots,r_7)^2}.
\]
Note that $b(x_1,\ldots,x_7)$ and $c(x_1,\ldots,x_7)$ are homogeneous symmetric polynomials.  Every term in $b(x_1,\ldots,x_7)$ has total degree $21$, and every term in $c(x_1,\ldots,x_7)$ has total degree $42$.
For $0 \leq k \leq 7$, we let $\sigma_k=\sigma_k(x_1,\ldots,x_7)$ be the elementary symmetric polynomial of degree $k$.
Then we can write
\[
c(x_1,\ldots,x_n) = \sums{(e_1,\ldots,e_7) \in \N^7 \\ e_1 + 2 e_2 + \ldots + 7 e_7=42} \lambda_{(e_1,\ldots,e_7)} \sigma_1^{e_1} \sigma_2^{e_2} \cdots \sigma_7^{e_7},
\]
where we use $\N$ to denote the set $\{0,1,2,\ldots\}$ of nonnegative integers, and where each $\lambda_{(e_1,\ldots,e_7)} \in \Ftwo.$
We have used a computer program to find these $\lambda_{(e_1,\ldots,e_7)}$ values.
There are $218$ indices $(e_1,\ldots,e_7)$ such that $\lambda_{(e_1,\ldots,e_7)}$ is nonzero (i.e., is equal to $1$).
These indices $(e_1,\ldots,e_7)$ for nonvanishing $\lambda_{(e_1,\ldots,e_7)}$ are listed on Tables \ref{Margaret}--\ref{Peter} in lexicographical order, which allows one easily to see that every nonzero $\lambda_{(e_1,\ldots,e_7)}$ has a positive value for at least one of $e_1$, $e_2$, $e_5$, or $e_6$.
Since $g_{F,a}(x)$ has no terms of degree $6$, $5$, $2$ or $1$, we know that $\sigma_k(r_1,\ldots,r_7)=0$ when $k \in \{1,2,5,6\}$.
This means that every term $\lambda_{(e_1,\ldots,e_7}) \sigma_1^{e_1} \cdots \sigma_7^{e_7}$ always vanishes when evaluated at $(r_1,\ldots,r_7)$, either because the coefficient $\lambda_{(e_1,\ldots,e_7)}$ is zero, or else because one of the accompanying symmetric polynomials evaluates to zero.
Thus $c(r_1,\ldots,r_7)=0$, and so $S=0$.
\end{proof}
We now examine the consequences of this calculation.
\begin{lemma}\label{Ernest}
Let $F$ be a finite field that is an even degree extension of $\Ftwo$ and let $a \in F$.  Suppose that the key polynomial $g_{F,a}(x)$ from Definition \ref{Alan} is separable and $R$ is its set of seven distinct roots in $\acFtwou$.  Then $R$ is a union of an even number of $\Pi_F$-orbits.
\end{lemma}
\begin{proof}
The key polynomial is self-conjugate-reciprocal, so $R$ is a union of $\Pi_F$-orbits by Proposition \ref{Helen}.
Let $t$ be the number of $\Pi_F$-orbits in this union, and since $g_{F,a}(x)$ is separable, the sum of the cardinalities of those orbits is $\card{R}=7$.
We let $S$ be as defined in Lemma \ref{Lawrence}, which tells us that $S=0$.  Therefore $\Tr_{H_F/\Ftwo}(S)=0$, but Proposition \ref{Timothy} shows that $\Tr_{H_F/\Ftwo}(S)=\binom{\card{R}+1}{2}+t \pmod{2}$, so $t$ must be even.
\end{proof}
\begin{proposition}\label{Joseph}
Let $F$ be a finite field of order $q$ that is an even degree extension of $\Ftwo$ and let $a \in F$.  Suppose that the key polynomial $g_{F,a}(x)$ from Definition \ref{Alan} is separable.  Then $g_{F,a}(x)$ does not have precisely four, six, or seven roots on $U_F$.  Furthermore,
\begin{enumerate}[(i).]
\item If $g_{F,a}(x)$ has zero roots on $U_F$, then the seven roots of $g_{F,a}(x)$ must be in two $\Pi_F$-orbits (either of sizes two and five or else of sizes three and four).  So $g_{F,a}(x)$ either has two roots in $F\setminus U_F$ and five roots on $U_{\F_{q^5}}\setminus F$ or else it has four roots in $\F_{q^2} \setminus F$ and three roots on $U_{\F_{q^3}} \setminus F$.
\item If $g_{F,a}(x)$ has one root on $U_F$, then the six remaining roots must either be in one $\Pi_F$-orbit of size six (which yields six roots in $\F_{q^3}\setminus (U_{\F_{q^3}}\cup F)$) or else they must be in three $\Pi_F$-orbits each of size two (which yields six roots in $F\setminus U_F$).
\item If $g_{F,a}(x)$ has two roots on $U_F$, then the five remaining roots must be in two $\Pi_F$-orbits of sizes two and three (which yields two roots in $F\setminus U_F$ and three roots on $U_{\F_{q^3}}\setminus F$).
\item If $g_{F,a}(x)$ has three roots on $U_F$, then the four remaining roots must be in one $\Pi_F$-orbit of size four (yielding four roots in $\F_{q^2}\setminus F$).
\item If $g_{F,a}(x)$ has five roots on $U_F$, then the two remaining roots must be in one $\Pi_F$-orbit of size two (yielding two roots in $F\setminus U_F$).
\end{enumerate}
\end{proposition}
\begin{proof}
From Lemma \ref{Ernest}, we know that the seven distinct roots of $g_{F,a}(x)$ are organized into an even number of $\Pi_F$-orbits.
Recall that the roots that lie on $U_F$ are precisely those in singleton orbits.
\begin{itemize}
\item So there cannot be seven roots on $U_F$, as this would mean that $R$ contains seven $\Pi_F$-orbits.
\item Nor can there be six roots on $U_F$, as this would force the seventh to be alone in its own orbit, making it a seventh a root on $U_F$.
\item Nor can there be four roots on $U_F$, as this would mean we have four singleton orbits, and the remaining three roots would need to be organized into an even number of orbits.  This means two orbits, so one of these remaining orbits would be of size one and thus place a fifth root on $U_F$.
\end{itemize}
The rest of the statements in this theorem are simple consequences of the fact that we must organize the seven distinct roots of $g_{F,a}(x)$ into an even number of $\Pi_F$-orbits, and the facts about the sizes of those orbits from Lemma \ref{David}.
\end{proof}

\subsection{Conclusion}

We combine the results of Lemma \ref{Cecilia} and Proposition \ref{Joseph} to prove all the claims in Theorem \ref{Mary} except those in the case where $F=\Ffour$.
If $F=\Ffour$, then our exponent $d=5$ is degenerate (a power of $2$ modulo $\card{F}-1$), in which case it is well known (see \cite[Theorem 1.1]{Katz-2012}) that $\{W_{F,d}(a): a \in \Fu\}=\{0,4\}=\{0,2\sqrt{\card{F}}\}$, and so Lemma \ref{William} shows that the set of counts of distinct roots on $U_F$ of key polynomials $g_{F,a}(x)$ with $a \in \Fu$ must be $\{1,3\}$.
Recall from Section \ref{Raphael} that Theorem \ref{George} is equivalent to Theorem \ref{Mary} by Lemma \ref{William}.
\section{Appendix}
Recall that in Subsection \ref{Nathan} we define the following polynomials in the ring $\Ftwo[x_1,\ldots,x_7]$:
\[
b(x)=\prod_{1 \leq i < j \leq 7} (x_i-x_j)
\]
and
\[
c(x)=b(x)^2 \sums{1 \leq i < j \leq 7} \frac{x_i x_j}{(x_i-x_j)^2},
\]
and since $c(x)$ is a symmetric polynomial, we let $\sigma_k(x_1,\ldots,x_7)$ denote the elementary symmetric polynomial of degree $k$ and write
\[
c(x_1,\ldots,x_n) = \sums{(e_1,\ldots,e_7) \in \N^7 \\ e_1 + 2 e_2 + \ldots + 7 e_7=42} \lambda_{(e_1,\ldots,e_7)} \sigma_1^{e_1} \sigma_2^{e_2} \cdots \sigma_7^{e_7},
\]
where each $\lambda_{e_1,\ldots,e_7} \in \Ftwo$.
The indices $(e_1,\ldots,e_7)$ such that $\lambda_{(e_1,\ldots,e_7)}=1$ are listed here on Tables \ref{Margaret}--\ref{Peter} in lexicographical order, showing that $(e_5,e_6)\not=(0,0)$ in the first $23$ rows, and $(e_1,e_2)\not=(0,0)$ subsequently.  Thus every nonzero $\lambda_{(e_1,\ldots,e_7)}$ has a positive value for at least one of $e_1$, $e_2$, $e_5$, or $e_6$.
\begin{table}[ht!]
\caption{Nonvanishing Terms of $c(x_1,\ldots,x_7)$}\label{Margaret}
\begin{center}
{\footnotesize \begin{tabular}{c|ccccccc||c|ccccccc}
 & \multicolumn{7}{c||}{$(e_1,\ldots,e_7)$ such} &  & \multicolumn{7}{c}{$(e_1,\ldots,e_7)$ such} \\ 
Term & \multicolumn{7}{c||}{that $\lambda_{(e_1,\ldots,e_7)}=1$} & Term & \multicolumn{7}{c}{that $\lambda_{(e_1,\ldots,e_7)}=1$} \\
No. & $e_1$ & $e_2$ & $e_3$ & $e_4$ & $e_5$ & $e_6$ & $e_7$ & No. & $e_1$ & $e_2$ & $e_3$ & $e_4$ & $e_5$ & $e_6$ & $e_7$ \\
\hline
  1 & 0 & 0 & 0 & 0 & 2 & 3 & 2 &  17 & 0 & 0 & 4 & 0 & 1 & 3 & 1 \\ 
  2 & 0 & 0 & 0 & 0 & 3 & 1 & 3 &  18 & 0 & 0 & 4 & 2 & 2 & 2 & 0 \\
  3 & 0 & 0 & 0 & 0 & 6 & 2 & 0 &  19 & 0 & 0 & 4 & 2 & 3 & 0 & 1 \\
  4 & 0 & 0 & 0 & 0 & 7 & 0 & 1 &  20 & 0 & 0 & 4 & 3 & 0 & 3 & 0 \\
  5 & 0 & 0 & 0 & 1 & 4 & 3 & 0 &  21 & 0 & 0 & 4 & 3 & 1 & 1 & 1 \\
  6 & 0 & 0 & 0 & 1 & 5 & 1 & 1 &  22 & 0 & 0 & 5 & 0 & 3 & 2 & 0 \\
  7 & 0 & 0 & 1 & 0 & 1 & 1 & 4 &  23 & 0 & 0 & 5 & 1 & 1 & 3 & 0 \\
  8 & 0 & 0 & 1 & 0 & 5 & 0 & 2 &  24 & 0 & 1 & 0 & 0 & 0 & 2 & 4 \\
  9 & 0 & 0 & 1 & 1 & 3 & 1 & 2 &  25 & 0 & 1 & 0 & 0 & 1 & 0 & 5 \\
 10 & 0 & 0 & 2 & 0 & 2 & 2 & 2 &  26 & 0 & 1 & 0 & 1 & 2 & 2 & 2 \\
 11 & 0 & 0 & 2 & 0 & 3 & 0 & 3 &  27 & 0 & 1 & 0 & 1 & 3 & 0 & 3 \\
 12 & 0 & 0 & 2 & 1 & 0 & 3 & 2 &  28 & 0 & 1 & 1 & 0 & 5 & 2 & 0 \\
 13 & 0 & 0 & 2 & 1 & 1 & 1 & 3 &  29 & 0 & 1 & 1 & 1 & 1 & 0 & 4 \\
 14 & 0 & 0 & 3 & 0 & 3 & 3 & 0 &  30 & 0 & 1 & 1 & 2 & 3 & 0 & 2 \\
 15 & 0 & 0 & 3 & 2 & 1 & 1 & 2 &  31 & 0 & 1 & 2 & 0 & 2 & 4 & 0 \\
 16 & 0 & 0 & 4 & 0 & 0 & 5 & 0 &  32 & 0 & 1 & 2 & 0 & 3 & 2 & 1
\end{tabular}}
\end{center}
\end{table}
\begin{table}
\caption{Nonvanishing Terms of $c(x_1,\ldots,x_7)$}
\begin{center}
{\footnotesize \begin{tabular}{c|ccccccc||c|ccccccc}
 & \multicolumn{7}{c||}{$(e_1,\ldots,e_7)$ such} &  & \multicolumn{7}{c}{$(e_1,\ldots,e_7)$ such} \\ 
Term & \multicolumn{7}{c||}{that $\lambda_{(e_1,\ldots,e_7)}=1$} & Term & \multicolumn{7}{c}{that $\lambda_{(e_1,\ldots,e_7)}=1$} \\
No. & $e_1$ & $e_2$ & $e_3$ & $e_4$ & $e_5$ & $e_6$ & $e_7$ & No. & $e_1$ & $e_2$ & $e_3$ & $e_4$ & $e_5$ & $e_6$ & $e_7$ \\
\hline
 33 & 0 & 1 & 2 & 2 & 0 & 2 & 2 &  65 & 1 & 0 & 0 & 0 & 0 & 1 & 5 \\
 34 & 0 & 1 & 2 & 2 & 1 & 0 & 3 &  66 & 1 & 0 & 0 & 0 & 4 & 0 & 3 \\
 35 & 0 & 1 & 3 & 0 & 1 & 2 & 2 &  67 & 1 & 0 & 0 & 1 & 2 & 1 & 3 \\
 36 & 0 & 1 & 3 & 1 & 3 & 2 & 0 &  68 & 1 & 0 & 1 & 0 & 2 & 0 & 4 \\
 37 & 0 & 1 & 3 & 3 & 1 & 0 & 2 &  69 & 1 & 0 & 1 & 0 & 4 & 3 & 0 \\
 38 & 0 & 1 & 4 & 1 & 0 & 4 & 0 &  70 & 1 & 0 & 1 & 1 & 0 & 1 & 4 \\
 39 & 0 & 1 & 4 & 1 & 1 & 2 & 1 &  71 & 1 & 0 & 1 & 2 & 2 & 1 & 2 \\
 40 & 0 & 2 & 0 & 0 & 2 & 0 & 4 &  72 & 1 & 0 & 2 & 0 & 2 & 3 & 1 \\
 41 & 0 & 2 & 0 & 0 & 4 & 3 & 0 &  73 & 1 & 0 & 2 & 2 & 0 & 1 & 3 \\
 42 & 0 & 2 & 0 & 1 & 0 & 1 & 4 &  74 & 1 & 0 & 3 & 0 & 0 & 3 & 2 \\
 43 & 0 & 2 & 0 & 2 & 2 & 1 & 2 &  75 & 1 & 0 & 3 & 0 & 4 & 2 & 0 \\
 44 & 0 & 2 & 2 & 0 & 1 & 1 & 3 &  76 & 1 & 0 & 3 & 1 & 2 & 3 & 0 \\
 45 & 0 & 2 & 2 & 0 & 5 & 0 & 1 &  77 & 1 & 0 & 3 & 2 & 2 & 0 & 2 \\
 46 & 0 & 2 & 2 & 1 & 3 & 1 & 1 &  78 & 1 & 0 & 3 & 3 & 0 & 1 & 2 \\
 47 & 0 & 2 & 2 & 2 & 2 & 0 & 2 &  79 & 1 & 0 & 4 & 0 & 2 & 2 & 1 \\
 48 & 0 & 2 & 2 & 3 & 0 & 1 & 2 &  80 & 1 & 0 & 4 & 1 & 0 & 3 & 1 \\
 49 & 0 & 2 & 3 & 0 & 3 & 0 & 2 &  81 & 1 & 1 & 0 & 0 & 4 & 2 & 1 \\  
 50 & 0 & 2 & 3 & 1 & 1 & 1 & 2 &  82 & 1 & 1 & 0 & 1 & 0 & 0 & 5 \\
 51 & 0 & 3 & 0 & 0 & 2 & 2 & 2 &  83 & 1 & 1 & 0 & 2 & 2 & 0 & 3 \\
 52 & 0 & 3 & 0 & 1 & 4 & 2 & 0 &  84 & 1 & 1 & 1 & 0 & 2 & 2 & 2 \\
 53 & 0 & 3 & 0 & 2 & 0 & 0 & 4 &  85 & 1 & 1 & 1 & 1 & 4 & 2 & 0 \\
 54 & 0 & 3 & 0 & 3 & 2 & 0 & 2 &  86 & 1 & 1 & 1 & 2 & 0 & 0 & 4 \\
 55 & 0 & 3 & 1 & 0 & 1 & 0 & 4 &  87 & 1 & 1 & 1 & 3 & 2 & 0 & 2 \\
 56 & 0 & 3 & 2 & 1 & 1 & 0 & 3 &  88 & 1 & 1 & 2 & 0 & 0 & 2 & 3 \\
 57 & 0 & 3 & 2 & 2 & 2 & 2 & 0 &  89 & 1 & 1 & 2 & 1 & 2 & 2 & 1 \\
 58 & 0 & 3 & 2 & 4 & 0 & 0 & 2 &  90 & 1 & 1 & 2 & 3 & 0 & 0 & 3 \\
 59 & 0 & 3 & 4 & 0 & 0 & 4 & 0 &  91 & 1 & 1 & 3 & 1 & 0 & 2 & 2 \\
 60 & 0 & 4 & 0 & 0 & 0 & 1 & 4 &  92 & 1 & 1 & 3 & 2 & 2 & 2 & 0 \\
 61 & 0 & 5 & 0 & 0 & 4 & 2 & 0 &  93 & 1 & 1 & 3 & 4 & 0 & 0 & 2 \\
 62 & 0 & 5 & 0 & 1 & 0 & 0 & 4 &  94 & 1 & 1 & 5 & 0 & 0 & 4 & 0 \\
 63 & 0 & 5 & 0 & 2 & 2 & 0 & 2 &  95 & 1 & 2 & 1 & 0 & 0 & 1 & 4 \\
 64 & 0 & 7 & 0 & 0 & 0 & 0 & 4 &  96 & 1 & 2 & 2 & 0 & 2 & 0 & 3
\end{tabular}}
\end{center}
\end{table}
\begin{table}
\caption{Nonvanishing Terms of $c(x_1,\ldots,x_7)$}
\begin{center}
{\footnotesize \begin{tabular}{c|ccccccc||c|ccccccc}
 & \multicolumn{7}{c||}{$(e_1,\ldots,e_7)$ such} &  & \multicolumn{7}{c}{$(e_1,\ldots,e_7)$ such} \\ 
Term & \multicolumn{7}{c||}{that $\lambda_{(e_1,\ldots,e_7)}=1$} & Term & \multicolumn{7}{c}{that $\lambda_{(e_1,\ldots,e_7)}=1$} \\
No. & $e_1$ & $e_2$ & $e_3$ & $e_4$ & $e_5$ & $e_6$ & $e_7$ & No. & $e_1$ & $e_2$ & $e_3$ & $e_4$ & $e_5$ & $e_6$ & $e_7$ \\
\hline
 97 & 1 & 2 & 2 & 1 & 0 & 1 & 3 & 129 & 2 & 2 & 1 & 0 & 1 & 0 & 4 \\
 98 & 1 & 3 & 0 & 0 & 0 & 0 & 5 & 130 & 2 & 2 & 1 & 0 & 3 & 3 & 0 \\
 99 & 1 & 3 & 1 & 0 & 4 & 2 & 0 & 131 & 2 & 2 & 1 & 2 & 1 & 1 & 2 \\
100 & 1 & 3 & 1 & 1 & 0 & 0 & 4 & 132 & 2 & 2 & 2 & 0 & 1 & 3 & 1 \\
101 & 1 & 3 & 1 & 2 & 2 & 0 & 2 & 133 & 2 & 2 & 2 & 2 & 3 & 0 & 1 \\
102 & 1 & 5 & 1 & 0 & 0 & 0 & 4 & 134 & 2 & 2 & 2 & 3 & 0 & 3 & 0 \\
103 & 2 & 0 & 0 & 0 & 0 & 2 & 4 & 135 & 2 & 2 & 2 & 3 & 1 & 1 & 1 \\
104 & 2 & 0 & 0 & 0 & 1 & 0 & 5 & 136 & 2 & 2 & 2 & 4 & 0 & 0 & 2 \\
105 & 2 & 0 & 0 & 0 & 2 & 5 & 0 & 137 & 2 & 2 & 3 & 0 & 3 & 2 & 0 \\
106 & 2 & 0 & 0 & 0 & 3 & 3 & 1 & 138 & 2 & 2 & 3 & 1 & 1 & 3 & 0 \\
107 & 2 & 0 & 0 & 2 & 0 & 3 & 2 & 139 & 2 & 2 & 4 & 0 & 0 & 4 & 0 \\
108 & 2 & 0 & 0 & 2 & 1 & 1 & 3 & 140 & 2 & 3 & 0 & 0 & 3 & 2 & 1 \\
109 & 2 & 0 & 1 & 1 & 3 & 3 & 0 & 141 & 2 & 3 & 0 & 2 & 1 & 0 & 3 \\
110 & 2 & 0 & 1 & 3 & 1 & 1 & 2 & 142 & 2 & 3 & 0 & 3 & 2 & 2 & 0 \\
111 & 2 & 0 & 2 & 1 & 0 & 5 & 0 & 143 & 2 & 3 & 0 & 5 & 0 & 0 & 2 \\
112 & 2 & 0 & 2 & 1 & 1 & 3 & 1 & 144 & 2 & 3 & 1 & 0 & 1 & 2 & 2 \\
113 & 2 & 0 & 2 & 2 & 0 & 2 & 2 & 145 & 2 & 3 & 1 & 1 & 3 & 2 & 0 \\
114 & 2 & 0 & 2 & 2 & 1 & 0 & 3 & 146 & 2 & 3 & 1 & 3 & 1 & 0 & 2 \\
115 & 2 & 0 & 3 & 0 & 1 & 2 & 2 & 147 & 2 & 3 & 2 & 1 & 1 & 2 & 1 \\
116 & 2 & 1 & 0 & 1 & 2 & 4 & 0 & 148 & 2 & 4 & 0 & 0 & 1 & 1 & 3 \\
117 & 2 & 1 & 0 & 1 & 3 & 2 & 1 & 149 & 2 & 4 & 0 & 0 & 4 & 2 & 0 \\
118 & 2 & 1 & 0 & 3 & 0 & 2 & 2 & 150 & 2 & 4 & 0 & 0 & 5 & 0 & 1 \\
119 & 2 & 1 & 0 & 3 & 1 & 0 & 3 & 151 & 2 & 4 & 0 & 1 & 3 & 1 & 1 \\
120 & 2 & 1 & 1 & 2 & 3 & 2 & 0 & 152 & 2 & 4 & 0 & 3 & 0 & 1 & 2 \\
121 & 2 & 1 & 1 & 4 & 1 & 0 & 2 & 153 & 2 & 4 & 1 & 0 & 3 & 0 & 2 \\
122 & 2 & 1 & 3 & 0 & 1 & 4 & 0 & 154 & 2 & 4 & 1 & 1 & 1 & 1 & 2 \\
123 & 2 & 2 & 0 & 0 & 3 & 0 & 3 & 155 & 2 & 5 & 0 & 1 & 1 & 0 & 3 \\
124 & 2 & 2 & 0 & 1 & 0 & 3 & 2 & 156 & 2 & 6 & 0 & 0 & 0 & 0 & 4 \\
125 & 2 & 2 & 0 & 1 & 1 & 1 & 3 & 157 & 3 & 0 & 0 & 1 & 2 & 3 & 1 \\
126 & 2 & 2 & 0 & 2 & 0 & 0 & 4 & 158 & 3 & 0 & 0 & 3 & 0 & 1 & 3 \\
127 & 2 & 2 & 0 & 2 & 2 & 3 & 0 & 159 & 3 & 0 & 1 & 0 & 2 & 2 & 2 \\
128 & 2 & 2 & 0 & 4 & 0 & 1 & 2 & 160 & 3 & 0 & 1 & 2 & 0 & 0 & 4
\end{tabular}}
\end{center}
\end{table}
\begin{table}
\caption{Nonvanishing Terms of $c(x_1,\ldots,x_7)$}\label{Peter}
\begin{center}
{\footnotesize \begin{tabular}{c|ccccccc||c|ccccccc}
 & \multicolumn{7}{c||}{$(e_1,\ldots,e_7)$ such} &  & \multicolumn{7}{c}{$(e_1,\ldots,e_7)$ such} \\ 
Term & \multicolumn{7}{c||}{that $\lambda_{(e_1,\ldots,e_7)}=1$} & Term & \multicolumn{7}{c}{that $\lambda_{(e_1,\ldots,e_7)}=1$} \\
No. & $e_1$ & $e_2$ & $e_3$ & $e_4$ & $e_5$ & $e_6$ & $e_7$ & No. & $e_1$ & $e_2$ & $e_3$ & $e_4$ & $e_5$ & $e_6$ & $e_7$ \\
\hline
161 & 3 & 0 & 1 & 2 & 2 & 3 & 0 & 190 & 3 & 4 & 1 & 0 & 0 & 0 & 4 \\
162 & 3 & 0 & 1 & 4 & 0 & 1 & 2 & 191 & 4 & 0 & 0 & 2 & 0 & 5 & 0 \\
163 & 3 & 0 & 2 & 0 & 0 & 2 & 3 & 192 & 4 & 0 & 0 & 2 & 1 & 3 & 1 \\
164 & 3 & 0 & 3 & 0 & 0 & 5 & 0 & 193 & 4 & 0 & 0 & 4 & 2 & 2 & 0 \\
165 & 3 & 0 & 3 & 2 & 2 & 2 & 0 & 194 & 4 & 0 & 0 & 4 & 3 & 0 & 1 \\
166 & 3 & 0 & 3 & 4 & 0 & 0 & 2 & 195 & 4 & 0 & 0 & 5 & 0 & 3 & 0 \\
167 & 3 & 0 & 5 & 0 & 0 & 4 & 0 & 196 & 4 & 0 & 0 & 5 & 1 & 1 & 1 \\
168 & 3 & 1 & 0 & 2 & 2 & 2 & 1 & 197 & 4 & 0 & 1 & 0 & 1 & 5 & 0 \\
169 & 3 & 1 & 0 & 4 & 0 & 0 & 3 & 198 & 4 & 0 & 1 & 3 & 1 & 3 & 0 \\
170 & 3 & 1 & 1 & 0 & 2 & 4 & 0 & 199 & 4 & 0 & 1 & 4 & 1 & 0 & 2 \\
171 & 3 & 1 & 1 & 2 & 0 & 2 & 2 & 200 & 4 & 0 & 3 & 0 & 1 & 4 & 0 \\
172 & 3 & 1 & 1 & 3 & 2 & 2 & 0 & 201 & 4 & 1 & 0 & 0 & 0 & 6 & 0 \\
173 & 3 & 1 & 1 & 5 & 0 & 0 & 2 & 202 & 4 & 1 & 0 & 0 & 1 & 4 & 1 \\
174 & 3 & 1 & 2 & 0 & 0 & 4 & 1 & 203 & 4 & 1 & 0 & 3 & 0 & 4 & 0 \\
175 & 3 & 1 & 3 & 1 & 0 & 4 & 0 & 204 & 4 & 1 & 0 & 3 & 1 & 2 & 1 \\
176 & 3 & 2 & 0 & 0 & 0 & 0 & 5 & 205 & 4 & 1 & 1 & 1 & 1 & 4 & 0 \\
177 & 3 & 2 & 0 & 0 & 2 & 3 & 1 & 206 & 4 & 2 & 0 & 1 & 1 & 3 & 1 \\
178 & 3 & 2 & 0 & 2 & 0 & 1 & 3 & 207 & 4 & 2 & 0 & 2 & 1 & 0 & 3 \\
179 & 3 & 2 & 1 & 0 & 0 & 3 & 2 & 208 & 4 & 2 & 1 & 0 & 1 & 2 & 2 \\
180 & 3 & 2 & 1 & 1 & 2 & 3 & 0 & 209 & 5 & 0 & 0 & 0 & 0 & 5 & 1 \\
181 & 3 & 2 & 1 & 3 & 0 & 1 & 2 & 210 & 5 & 0 & 0 & 3 & 0 & 3 & 1 \\
182 & 3 & 2 & 2 & 0 & 2 & 2 & 1 & 211 & 5 & 0 & 0 & 4 & 0 & 0 & 3 \\
183 & 3 & 2 & 2 & 1 & 0 & 3 & 1 & 212 & 5 & 0 & 1 & 1 & 0 & 5 & 0 \\
184 & 3 & 3 & 0 & 0 & 0 & 2 & 3 & 213 & 5 & 0 & 1 & 2 & 0 & 2 & 2 \\
185 & 3 & 3 & 0 & 1 & 2 & 2 & 1 & 214 & 5 & 0 & 2 & 0 & 0 & 4 & 1 \\
186 & 3 & 3 & 0 & 3 & 0 & 0 & 3 & 215 & 5 & 1 & 0 & 1 & 0 & 4 & 1 \\
187 & 3 & 3 & 1 & 1 & 0 & 2 & 2 & 216 & 5 & 2 & 0 & 0 & 0 & 2 & 3 \\
188 & 3 & 4 & 0 & 0 & 2 & 0 & 3 & 217 & 6 & 0 & 0 & 0 & 0 & 6 & 0 \\
189 & 3 & 4 & 0 & 1 & 0 & 1 & 3 & 218 & 6 & 0 & 0 & 0 & 1 & 4 & 1
\end{tabular}}
\end{center}
\end{table}
\section*{Acknowledgment}
The authors thank the anonymous reviewers for their comments, which helped improve the paper.


\begin{thebibliography}{DFHR06}

\bibitem[AKL15]{Aubry-Katz-Langevin}
Yves Aubry, Daniel~J. Katz, and Philippe Langevin.
\newblock Cyclotomy of {W}eil sums of binomials.
\newblock {\em J. Number Theory}, 154:160--178, 2015.

\bibitem[CCD00]{Canteaut-Charpin-Dobbertin}
Anne Canteaut, Pascale Charpin, and Hans Dobbertin.
\newblock Binary {$m$}-sequences with three-valued crosscorrelation: a proof of
  {W}elch's conjecture.
\newblock {\em IEEE Trans. Inform. Theory}, 46(1):4--8, 2000.

\bibitem[CD96]{Cusick-Dobbertin}
Thomas~W. Cusick and Hans Dobbertin.
\newblock Some new three-valued crosscorrelation functions for binary
  {$m$}-sequences.
\newblock {\em IEEE Trans. Inform. Theory}, 42(4):1238--1240, 1996.

\bibitem[DFHR06]{Dobbertin-Felke-Helleseth-Rosendahl}
Hans Dobbertin, Patrick Felke, Tor Helleseth, and Petri Rosendahl.
\newblock Niho type cross-correlation functions via {D}ickson polynomials and
  {K}loosterman sums.
\newblock {\em IEEE Trans. Inform. Theory}, 52(2):613--627, 2006.

\bibitem[Gam86a]{Games-Sequences}
Richard~A. Games.
\newblock The geometry of {$m$}-sequences: three-valued crosscorrelations and
  quadrics in finite projective geometry.
\newblock {\em SIAM J. Algebraic Discrete Methods}, 7(1):43--52, 1986.

\bibitem[Gam86b]{Games-Quadrics}
Richard~A. Games.
\newblock The geometry of quadrics and correlations of sequences.
\newblock {\em IEEE Trans. Inform. Theory}, 32(3):423--426, 1986.

\bibitem[Gol68]{Gold}
Robert Gold.
\newblock Maximal recursive sequences with 3-valued recursive cross-correlation
  functions.
\newblock {\em IEEE Trans. Inform. Theory}, 14(1):154--156, 1968.

\bibitem[Hel76]{Helleseth-1976}
Tor Helleseth.
\newblock Some results about the cross-correlation function between two maximal
  linear sequences.
\newblock {\em Discrete Math.}, 16(3):209--232, 1976.

\bibitem[Hel78]{Helleseth-1978}
Tor Helleseth.
\newblock A note on the cross-correlation function between two binary maximal
  length linear sequences.
\newblock {\em Discrete Math.}, 23(3):301--307, 1978.

\bibitem[HK98]{Helleseth-Kumar}
Tor Helleseth and P.~Vijay Kumar.
\newblock Sequences with low correlation.
\newblock In V.~S. Pless, W.~C. Huffman, and R.~A. Brualdi, editors, {\em
  Handbook of coding theory}, volume~II, chapter~21, pages 1765--1853.
  North-Holland, Amsterdam, 1998.

\bibitem[HX01]{Hollmann-Xiang}
Henk D.~L. Hollmann and Qing Xiang.
\newblock A proof of the {W}elch and {N}iho conjectures on cross-correlations
  of binary {$m$}-sequences.
\newblock {\em Finite Fields Appl.}, 7(2):253--286, 2001.

\bibitem[Kat12]{Katz-2012}
Daniel~J. Katz.
\newblock Weil sums of binomials, three-level cross-correlation, and a
  conjecture of {H}elleseth.
\newblock {\em J. Combin. Theory Ser. A}, 119(8):1644--1659, 2012.

\bibitem[Kat19]{Katz-2019}
Daniel~J. Katz.
\newblock Weil sums of binomials: properties, applications and open problems.
\newblock In Kai-Uwe Schmidt and Arne Winterhof, editors, {\em Combinatorics
  and Finite Fields: Difference Sets, Polynomials, Pseudorandomness and
  Applications}, volume~23 of {\em Radon Ser. Comput. Appl. Math.}, pages
  109--134. De Gruyter, Berlin, Boston, 2019.

\bibitem[LZ19]{Li-Zeng}
Nian Li and Xiangyong Zeng.
\newblock A survey on the applications of {N}iho exponents.
\newblock {\em Cryptogr. Commun.}, 11(3):509--548, 2019.

\bibitem[Nih72]{Niho}
Yoji Niho.
\newblock {\em Multi-valued cross-correlation function between two maximal
  linear recursive sequences}.
\newblock PhD thesis, University of Southern California, Los Angeles, 1972.

\bibitem[Ros06]{Rosendahl}
Petri Rosendahl.
\newblock A generalization of {N}iho's theorem.
\newblock {\em Des. Codes Cryptogr.}, 38(3):331--336, 2006.

\bibitem[Tra70]{Trachtenberg}
Herbert~Mitchell Trachtenberg.
\newblock {\em On the cross-correlation functions of maximal linear sequences}.
\newblock PhD thesis, University of Southern California, Los Angeles, 1970.

\bibitem[XLZH16]{Xia-Li-Zeng-Helleseth}
Yongbo Xia, Nian Li, Xiangyong Zeng, and Tor Helleseth.
\newblock An open problem on the distribution of a {N}iho-type
  cross-correlation function.
\newblock {\em IEEE Trans. Inform. Theory}, 62(12):7546--7554, 2016.

\end{thebibliography}
\end{document}